\documentclass[11pt]{article}
\def\thetitle{Minimum degree conditions for graph rigidity}

\usepackage{graphicx}

\usepackage{amsmath,amssymb}

\usepackage[usenames,dvipsnames,svgnames,table]{xcolor}
\definecolor{CombinatoricaAqua}{HTML}{00698C}
\definecolor{CombinatoricaBlue}{HTML}{3A3293}
\definecolor{CombinatoricaBrown}{HTML}{66220C}
\definecolor{CombinatoricaRed}{HTML}{DF2A27}
\definecolor{HarvardCrimson}{rgb}{0.6471, 0.1098, 0.1882}
\definecolor{DAGreen}{HTML}{339900}

\makeatletter
\let\reftagform@=\tagform@
\def\tagform@#1{\maketag@@@
	{(\ignorespaces\textcolor{CombinatoricaBrown}{#1}\unskip\@@italiccorr)}}
\renewcommand{\eqref}[1]{\textup{\reftagform@{\ref{#1}}}}
\makeatother

\usepackage[backref=page]{hyperref}
\hypersetup{%
	unicode,
	pdfencoding=auto,
	pdfauthor={Peleg Michaeli},
	pdftitle={\thetitle},
	pdfsubject={},
	pdfkeywords={},
	colorlinks=true,
	citecolor=CombinatoricaBlue,
	linkcolor=CombinatoricaAqua,
	anchorcolor=CombinatoricaBrown,
	urlcolor=HarvardCrimson}

\usepackage{amsthm}
\usepackage{thmtools}
\usepackage{bbm}
\usepackage{enumitem}
\usepackage[capitalize]{cleveref}
\Crefname{mainthm}{Theorem}{Theorems}
\Crefname{fact}{Fact}{Facts}
\Crefname{claim}{Claim}{Claims}


\declaretheoremstyle[
spaceabove=\topsep, spacebelow=\topsep,
headfont=\color{CombinatoricaBrown}\normalfont\bfseries,
bodyfont=\itshape,
]{thm}
\declaretheoremstyle[
spaceabove=\topsep, spacebelow=\topsep,
headfont=\color{CombinatoricaBrown}\normalfont\bfseries,
bodyfont=\normalfont,
]{dfn}
\declaretheoremstyle[
spaceabove=0.5\topsep, spacebelow=0.5\topsep,
headfont=\color{CombinatoricaBrown}\normalfont\bfseries,
bodyfont=\normalfont,
]{rmk}

\declaretheorem[style=thm,name=Theorem]{mainthm}
\declaretheorem[style=thm,parent=section]{theorem}
\declaretheorem[style=thm,sibling=theorem]{lemma}
\declaretheorem[style=thm,sibling=theorem]{corollary}
\declaretheorem[style=thm,sibling=theorem]{claim}
\declaretheorem[style=thm,sibling=theorem]{proposition}

\declaretheorem[style=thm,sibling=mainthm]{conjecture}

\usepackage[nobysame,msc-links,non-sorted-cites]{amsrefs}

\renewcommand{\eprint}[1]{\href{https://arxiv.org/abs/#1}{arXiv:#1}}

\BibSpec{book}{%
	+{}  {\PrintPrimary}                {transition}
	+{,} { \textbf}                     {title} 
	+{.} { }                            {part}
	+{:} { \textit}                     {subtitle}
	+{,} { \PrintEdition}               {edition}
	+{}  { \PrintEditorsB}              {editor}
	+{,} { \PrintTranslatorsC}          {translator}
	+{,} { \PrintContributions}         {contribution}
	+{,} { }                            {series}
	+{,} { \voltext}                    {volume}
	+{,} { }                            {publisher}
	+{,} { }                            {organization}
	+{,} { }                            {address}
	+{,} { \PrintDateB}                 {date}
	+{,} { }                            {status}
	+{}  { \parenthesize}               {language}
	+{}  { \PrintTranslation}           {translation}
	+{;} { \PrintReprint}               {reprint}
	+{.} { }                            {note}
	+{.} {}                             {transition}
	+{}  {\SentenceSpace \PrintReviews} {review}
}
\BibSpec{incollection}{%
  +{}  {\PrintAuthors}                {author}
  +{,} { \textit}                     {title}
  +{.} { }                            {part}
  +{:} { \textit}                     {subtitle}
  +{,} { \PrintContributions}         {contribution}
  +{,} { \PrintConference}            {conference}
  +{}  {\PrintBook}                   {book}
  +{,} { }                            {booktitle}
	+{}  { \PrintEditorsB}              {editor}
	+{,} { }                            {publisher}
  +{,} { \PrintDateB}                 {date}
  +{,} { pp.~}                        {pages}
  +{,} { }                            {status}
  +{,} { \PrintDOI}                   {doi}
  +{,} { available at \eprint}        {eprint}
  +{}  { \parenthesize}               {language}
  +{}  { \PrintTranslation}           {translation}
  +{;} { \PrintReprint}               {reprint}
  +{.} { }                            {note}
  +{.} {}                             {transition}
  +{}  {\SentenceSpace \PrintReviews} {review}
}

\makeatletter
\renewcommand{\PrintNames@a}[4]{%
	\PrintSeries{\name}
	{#1}
	{}{ and \set@othername}
	{,}{ \set@othername}
	{}{ and \set@othername}
	{#2}{#4}{#3}%
}
\makeatother

\makeatletter
\def\mathcolor#1#{\@mathcolor{#1}}
\def\@mathcolor#1#2#3{%
	\protect\leavevmode
	\begingroup
	\color#1{#2}#3%
	\endgroup
}
\makeatother
\definecolor{Red}{rgb}{0.618,0,0}
\definecolor{Blue}{rgb}{0,0,1}
\definecolor{Green}{rgb}{0,0.298,0}

\usepackage{sectsty}
\sectionfont{\color{CombinatoricaBrown}}
\subsectionfont{\color{CombinatoricaBrown}}
\subsubsectionfont{\color{CombinatoricaBrown}}

\usepackage{soul}
\soulregister\ref7

\usepackage{pifont}
\usepackage{calc}

\usepackage[a4paper]{geometry}
\title{\thetitle}
\author{
  Michael Krivelevich\thanks{
    School of Mathematical Sciences,
    Tel Aviv University,
    Tel Aviv 6997801, Israel.
    Email: \href{mailto:krivelev@tauex.tau.ac.il}
                {\tt krivelev@tauex.tau.ac.il}.
    Research supported in part by NSF-BSF grant 2023688.
  }
  \and
  Alan Lew\thanks{
    Department of Mathematical Sciences, Mellon College of Science, Carnegie Mellon University, Pittsburgh, PA 15213,
USA.
    Email: \href{mailto:alanlew@andrew.cmu.edu}
                {\tt alanlew@andrew.cmu.edu}.
  }
  \and
  Peleg Michaeli\thanks{
    Mathematical Institute,
    University of Oxford,
    Oxford, UK.
    Email: \href{mailto:peleg.michaeli@maths.ox.ac.uk}
                {\tt peleg.michaeli@maths.ox.ac.uk}.
    Research supported by ERC Advanced Grant 883810.
    For the purpose of Open Access, the author has applied a CC BY public
    copyright licence to any Author Accepted Manuscript version arising from
    this submission.
  }
}

\makeatletter
\def\namedlabel#1#2{\begingroup
  #2%
  \def\@currentlabel{#2}%
  \phantomsection\label{#1}\endgroup
}
\makeatother

\usepackage{mleftright}
\mleftright
\newcommand{\defn}[1]{{\bfseries #1}}

\newcommand{\midd}{\ \middle\vert\ }

\newcommand{\eps}{\varepsilon}
\renewcommand{\phi}{\varphi}

\newcommand{\RR}{\mathbb{R}}

\newcommand{\cB}{\mathcal{B}}

\newcommand{\cE}{\mathcal{E}}

\newcommand{\sR}{\mathsf{R}}

\newcommand{\sm}{\smallsetminus}
\newcommand{\es}{\varnothing}

\newcommand{\floor}[1]{\left\lfloor{#1}\right\rfloor}

\DeclareMathOperator{\rank}{rank}

\newcommand{\vect}{\mathbf}

\newcommand{\y}{\vect{y}}

\newcommand{\pr}[0]{\mathbb{P}}
\newcommand{\E}[0]{\mathbb{E}}

\newcommand{\whp}[0]{\textbf{whp}}

\newcommand{\Dist}[1]{\mathsf{#1}}
\newcommand{\Bin}{\Dist{Bin}}

\newcommand{\Cl}{\mathsf{C}}
\newcommand{\OK}{OK}
\newcommand{\p}{\vect{p}} 
\newcommand{\R}{\sR} 

\usepackage{tabto}

\usepackage{comment}

\usepackage{subcaption}
\usepackage{tikz,pgfplots}
\pgfplotsset{compat=1.16}
\usetikzlibrary{calc}

\begin{document}

\maketitle

\begin{abstract}
  We study minimum degree conditions that guarantee that an $n$-vertex graph is rigid in $\mathbb{R}^d$.
  For small values of $d$,
  we obtain a tight bound: for $d = O(\sqrt{n})$,
  every $n$-vertex graph with minimum degree at least $(n+d)/2 - 1$ is rigid in $\mathbb{R}^d$.
  For larger values of $d$,
  we achieve an approximate result: for $d = O(n/{\log^2}{n})$,
  every $n$-vertex graph with minimum degree at least $(n+2d)/2 - 1$ is rigid in $\mathbb{R}^d$.
  This bound is tight up to a factor of two in the coefficient of $d$.

  As a byproduct of our proof,
  we also obtain the following result, which may be of independent interest: for $d = O(n/{\log^2}{n})$, every $n$-vertex graph with minimum degree at least $d$
  has pseudoachromatic number at least $d+1$;
  namely, the vertex set of such a graph can be partitioned into $d+1$ subsets
  such that there is at least one edge between each pair of subsets.
  This is tight.
\end{abstract}

\section{Introduction}

A \defn{$d$-dimensional framework}
is a pair $(G,\p)$ consisting of a graph $G=(V,E)$ and an embedding $\p:V\to \RR^d$.
Given such a framework, one may ask:
is there a continuous motion of the vertices,
starting from the positions prescribed by $\p$,
that preserves the distance between all pairs of adjacent vertices of $G$?
The answer is always trivially positive:
translations and rotations of the whole graph preserve all such distances.
We say that the framework $(G,\p)$ is \defn{rigid} if,
aside from these trivial motions,
there is no continuous motion of the vertices that preserves the distance between all pairs of adjacent vertices.
An embedding $\p$ is \defn{generic} if its $d|V|$ coordinates are algebraically independent over the rationals.
Asimow and Roth showed in \cite{AR78} that, for generic $\p$, the rigidity of $(G,\p)$ depends only on the graph $G$.
We say that $G$ is \defn{$d$-rigid} if $(G,\p)$ is rigid for a generic embedding $\p:V\to \RR^d$.

In practice, it is often easier to work with the related notion of infinitesimal rigidity,
defined as follows.
Let $(G,\p)$ be a $d$-dimensional framework.
Let $C=\{ (v,i):\, v\in V,\, i\in[d]\}$.  
The \defn{rigidity matrix} of $(G,\p)$, denoted by $\R(G,\p)$, is an $E\times C$ matrix defined by
\[
    \R(G,\p)_{e,(u,i)}= \begin{cases}
        \p(u)_i-\p(v)_i & \text{if } e=\{u,v\} \text{ for some } v\in V,\\
        0 & \text{if } u\notin e
    \end{cases}
\]
for all $e\in E$ and $(u,i)\in C$. 
In other words, the rows of $\R(G,\p)$ are indexed by the edges of $G$,
its columns are indexed by the vertices, where each vertex is assigned $d$ consecutive columns,
and the row vector indexed by an edge $e=\{u,v\}\in E$ is supported on the $2d$ columns associated with $u$ and $v$,
where it is equal to $\p(u)-\p(v)$ and $\p(v)-\p(u)$ respectively.
For simplicity, from now on we will always assume that the image of $\p$ is $d$-dimensional
(that is, there is no hyperplane in $\RR^d$ containing all points in the image).
In particular, we assume that $G$ has at least $d+1$ vertices.
It is known that the rank of the rigidity matrix is always at most $d|V|-\binom{d+1}{2}$ (see \cite{AR78}).
We say that $(G,\p)$ is \defn{infinitesimally rigid} if this bound is achieved,
that is, if $\rank(\R(G,\p))=d|V|-\binom{d+1}{2}$.
Infinitesimal rigidity of a framework $(G,\p)$ implies its rigidity \cite{Gluck1974}.
Moreover, for generic embeddings, both notions are equivalent \cite{AR78}. For more background on the theory of rigid graphs, see for example \cites{schulze2017rigidity,GSS1993,jordan2016rigidity}.

In this paper, we study minimum degree sufficient conditions for rigidity of graphs.
For a graph $G$, let $\delta(G)$ denote its minimum degree.
A graph $G$ on $n>d$ vertices is said to be \defn{$d$-connected}
if it remains connected whenever fewer than $d$ vertices are removed.
It is a well-known fact that every $d$-rigid graph is $d$-connected\footnote{%
Indeed, otherwise, let $S$ be a set of size smaller than $d$ such that $G[V\sm S]$ is not connected,
and let $U$ be a connected component of $G[V\sm S]$.
Then, we can rotate the vertices of $U$ around $S$ independently of the rest of the graph
while preserving the lengths of all edges in $G$ --- a contradiction to the $d$-rigidity of $G$.}.
Another well-known and easy to prove fact is that every graph with minimum degree of at least $(n+d)/2-1$
is $d$-connected.
This statement is sharp.
Indeed, for every $m,n,d$ with $0\le d\le \min\{m,n\}$, let $\OK_{m,n,d}$ denote the union
of an $m$-clique and an $n$-clique that intersect in exactly $d$ vertices.
Let $d\le n$ be integers with the same parity, and let $m=(n+d)/2\ge d$.
Note that $\delta(\OK_{m,m,d})=(n+d)/2-1$.
On the other hand, $\OK_{m,m,d}$ is $d$-connected, but not $(d+1)$-connected.
(In fact, using standard rigidity results, one can show that $\OK_{m,n,d}$ is $d$-rigid  for every $m,n\ge d$.)

One might wonder whether every graph $G$ with $\delta(G)\ge (n+d)/2-1$ is also $d$-rigid
(and indeed, we will show that this is the case when $d$ is small relative to $n$).
However, note that by the definition of the rigidity matrix,
in order for an $n$-vertex graph $G=(V,E)$ to be $d$-rigid, one must have $|E|\ge d n-\binom{d+1}{2}$.
This is achieved if $\delta(G) \ge 2d-d(d+1)/n$.
Therefore, we conjecture the following (see also \cref{fig:conj}):

\begin{conjecture}\label{conj:main}
  Let $n>d\ge 1$.
  If $G$ is an $n$-vertex graph with
  \[\delta(G)\ge\max\left\{\frac{n+d}{2}-1,2d-\frac{d(d+1)}{n}\right\},\]
  then $G$ is $d$-rigid.
\end{conjecture}

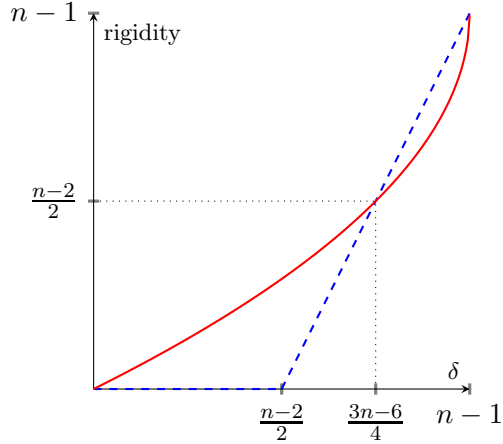
\begin{figure}
\captionsetup{width=0.879\textwidth,font=small}
\begin{center}
  \begin{tikzpicture}
    \begin{axis}[
        clip=false,
        x=5cm,y=5cm,
        axis lines=middle,
        xmin=0,
        xmax=1,
        ymin=0,
        ymax=1,
        xlabel={\footnotesize $\delta$},
        ylabel={\footnotesize rigidity},
        xtick={0,1/2,3/4,1},
        xticklabels={,$\frac{n-2}{2}$,$\frac{3n-6}{4}$,$n-1$},
        ytick={0,1/2,1},
        yticklabels={,$\frac{n-2}{2}$,$n-1$},
        tick style={very thick},
        legend style={
            at={(rel axis cs:0.25,0.75)},
            anchor=north west,draw=none,inner sep=0pt}
      ]
      \addplot[domain={0:1/2},blue,dashed,thick,samples=10] {0};
      \addplot[domain={1/2:1},blue,dashed,thick,samples=10] {2*x-1};
      \addplot[domain={0:1},red,thick,samples=200] {1-sqrt(1-x)};
      \draw[dotted] (axis cs:0,0.5) -- (axis cs:0.75,0.5) -- (axis cs:0.75,0);
    \end{axis}
  \end{tikzpicture}
\end{center}
\caption{ We define the \defn{rigidity} of a graph $G=(V,E)$ to be the maximal $d$ for which $G$ is $d$-rigid. 
  In the above diagram, the $x$-axis shows the minimum degree, while the $y$-axis displays the
  rigidity of the graph. The red solid line represents the bound $\delta(G)\ge 2d-d(d+1)/n$, determined by the necessary condition $|E|\ge dn-\binom{d+1}{2}$. 
  The blue dashed line
  represents the bound $\delta(G)\ge (n+d)/2-1$, determined by the necessary condition of $d$-connectivity. These lines intersect when $\delta(G) = (3n-6)/4$. The conjectured
  lower bound on the graph's rigidity is the minimum of these two lines.
}\label{fig:conj}
\end{figure}

Note that \cref{conj:main} is a stronger version of Conjecture~1.16 in~\cite{KLM23+}.
It also includes, as a special case, a recently-proved~\cite{HMN24++} conjecture by the second author
(see \cite{CN23}*{Conjecture~3.4})
on the rigidity of the hyperoctahedral graph.
In \cite[Conjecture 38]{Jac16}, Jackson, Jord\'an, and Tanigawa made an analogous conjecture
(in the special case of fixed $d$ and growing $n$) in the context of low rank matrix completability
(which can be seen as a generalisation of rigidity to the setting of graphs with self-loops, see~\cite{Jac14}).

In \cite{KLM23+}, we proved some weak versions of \cref{conj:main}.
Specifically, we showed that if an $n$-vertex graph $G$ satisfies $\delta(G)\ge n/2+ k$,
then $G$ is $d$-rigid for $d=2k/(3\log{n})$ (\cite[Theorem~1.14]{KLM23+}).
Additionally, if $k=O(\sqrt{n}/\log{n})$, then  $G$ is $(k+1)$-rigid
(\cite[Theorem 1.15]{KLM23+}).
Here, we make further progress towards \cref{conj:main}. First, we show that the conjecture holds for small values of $d$.

\begin{theorem}\label{thm:exact}
  Let $n\ge 2$ and $0\le d\le (\sqrt{8n-15}-1)/4$.
  If $G$ is an $n$-vertex graph with $\delta(G)\ge (n+d)/2-1$, 
  then $G$ is $d$-rigid.
\end{theorem}

The proof relies on a recent argument by Vill\'anyi~\cite{Vil23+}
which shows that a graph $G$ with minimum degree at least $d(d+1)$
has a vertex whose neighbours form a clique in the $d$-closure of $G$
(see \cref{sec:closure} for the definition of a $d$-closure).

Handling the case of larger values of $d$ turns out to be significantly more challenging.
Here, we obtain the following result.

\begin{theorem}\label{thm:approx}
  There exist $c>0$ and $n_0\ge 2$ for which the following holds.
  For every $n\ge n_0$ and $1\le d\le cn/{\log^2}n$,
  every $n$-vertex graph $G$ with $\delta(G)\ge (n+2d)/2-1$ is $d$-rigid.
\end{theorem}

The bound on the rigidity of $G$ in  \cref{thm:approx} is optimal up to a factor of two.
In order to prove \cref{thm:approx}, we divide the argument into two cases,
depending on whether the graph $G$ is close to being a bipartite graph or not.
In case the graph is close to being bipartite,
we apply a result by Lew, Nevo, Peled, and Raz~\cite{LNPR23+}
(see also~\cite{KLM23+}),
which provides a sufficient condition for the rigidity of a graph
in terms of the existence of a ``$d$-rigid partition''
(a partition of the vertex set of the graph satisfying certain connectivity properties;
see \cref{sec:rigid_partition}). 
When the graph is far from bipartite,
we first apply the Regularity Lemma to find a linearly large $d$-rigid subgraph of $G$
(using again the method of rigid partitions).
We continue by showing
how the existence of this large $d$-rigid subgraph implies the rigidity of the whole graph.

As a corollary to the proof of \cref{thm:approx},
we obtain a related result concerning the pseudoachromatic number of graphs with a given minimum degree.
A \defn{pseudocomplete colouring} of a graph is a colouring of its vertices
--- not necessarily proper ---
such that each pair of colours appears together on at least one edge. 
Equivalently, it is a partition of the vertex set in which every two distinct parts are connected by at least one edge.
The \defn{pseudoachromatic number} of a graph $G$,
denoted $\psi_s(G)$,
is the maximum number of colours in a pseudocomplete colouring of that graph.
This concept was first introduced by Gupta~\cite{Gup69},
and was later used in the study of the {\em Hadwiger number} of a graph; 
see, e.g.,~\cites{BCE80,Kos82,Tho01}.

\begin{theorem}\label{thm:pseudo}
  There exist $c>0$ and $n_0\ge 2$ for which the following holds.
  For every $n\ge n_0$ and $1\le d\le cn/{\log^2}n$,
  every $n$-vertex graph $G$ with $\delta(G)\ge d$
  has $\psi_s(G)\ge d+1$.
\end{theorem}

We remark that \cref{thm:pseudo} is sharp in two senses.
First, the lower bound on $\psi_s(G)$ cannot be improved in terms of $d$.
Indeed, consider the complete bipartite graph $K_{d,n}$ with sides of sizes $d$ and $n\ge d$.
Then, $\delta(K_{d,n})=d$,
and any pseudocomplete colouring of $K_{d,n}$
must have at most one colour class without an element from the smaller side.
Indeed, if two distinct colour classes do not contain an element from the smaller side,
they are not connected by an edge.
Thus, $\psi_s(K_{d,n})\le d+1$ (and, in fact, $\psi_s(K_{d,n})=d+1$).
Second, the upper bound on $d$ cannot be significantly improved in terms of $n$.
Indeed, the binomial random graph $G(n,1/2)$ has $\delta(G(n,1/2))\ge (1/2-o(1))n$ and $\psi_s(G(n,1/2))=O(n/\sqrt{\log{n}})$
with high probability\footnote{\whp{}; that is, with probability tending to $1$ as $n\to\infty$}
(see~\cites{BCE80,HKRS07}).

~

The paper is organised as follows. 
In \cref{sec:prelims} we introduce some preliminary results on rigidity of graphs,
and define our main tool (rigid partitions).
In \cref{sec:small} we prove \cref{thm:exact},
and in \cref{sec:large} we prove \cref{thm:approx,thm:pseudo}.

\paragraph{Notation}
Let $G=(V,E)$ be a graph.
Write $V(G)=V$ and $E(G)=E$.
For two (not necessarily disjoint) vertex sets $A$, $B$,
we let $E_G(A, B)$
be the set of edges having one endpoint in $A$ and the other in $B$,
and write $E_G(A)=E_G(A,A)$.
We denote by $N_G(A)$ the
{\em external} neighbourhood of $A$,
that is, the set of all vertices in $V\sm A$
that have a neighbour in $A$.
Furthermore, in the above notation, we often replace $\{v\}$ with $v$ for abbreviation.
The degree of a vertex $v\in V$,
denoted by $\deg_G(v)$,
is its number of incident edges.
The degree of $v$ ``into'' $A$,
denoted $\deg_G(v,A)$,
is $|E_G(v,A)|$.
When the underlying graph is clear from the context, we drop the subscript $G$ and simply write $E(A,B)$, $E(A)$,  $N(A)$ etc.
We write $\Delta(G)$ for the maximum degree of $G$.
Throughout the paper, all logarithms are in the natural basis.

\section{Preliminaries}\label{sec:prelims}

\subsection{The $d$-closure of a graph}\label{sec:closure}

Let $n\ge d+1$, and let $G=(V,E)$ be an $n$-vertex graph.
 For $\{u,v\}\in\binom{V}{2}$, let $G+\{u,v\}$ be the graph obtained from $G$ by adding the edge $\{u,v\}$ (where, if $\{u,v\}\in E$, then $G+\{u,v\}=G$).

Let $\p:V\to\RR^d$ be a generic embedding. 
The \defn{$d$-rigidity closure} (or, in short, \defn{$d$-closure}) of $G$ is the graph $\Cl_d(G)$ on vertex set $V$ and edge set
\[
  E(\Cl_d(G))= \left\{
    \{u,v\}\in \binom{V}{2}\midd \rank(\R(G,\p))=\rank(\R(G+\{u,v\},\p))
  \right\}.
\]

Note that $\Cl_d(G)$ does not depend on $\p$, provided $\p$ is generic.
We say that $G$ is \defn{$d$-closed} if it is its own $d$-closure.
Observe that a graph is $d$-rigid if and only if its $d$-closure is $d$-rigid,
which is if and only if its $d$-closure is complete.

We will need the following simple lemma about $d$-closed graphs (see e.g. \cite{LNPR23}*{Observation~2.1}).

\begin{lemma}\label{lem:extension_in_closure}
    Let $G=(V,E)$ be a $d$-closed graph. Let $U$ be a clique of size $d$ in $G$, and let
    $v,w\in V\sm U$ be two distinct vertices such that both $v$ and $w$ are adjacent to all the vertices in $U$. Then, $\{v,w\}\in E$.
\end{lemma}

As an immediate consequence, we obtain the following well known property (see e.g.~\cite{TW1985}).

\begin{lemma}[$0$-extension property]\label{lem:extension}
    Let $G=(V,E)$ be a $d$-rigid graph, and let $v\notin V$. Let $G'$ be a graph obtained from $G$ by adding the vertex $v$ and adding $d$ edges between $v$ and $V$. Then, $G'$ is $d$-rigid.
\end{lemma}

\subsection{Rigid partitions}\label{sec:rigid_partition}

Let $V$ be a finite set. We say that $V_1,\ldots,V_m\subseteq V$ is a \defn{colouring} of $V$ if $V_i\cap V_j=\es$ for all $1\le i<j\le m$ and $V_1\cup\cdots\cup V_m= V$. If, in addition, $V_i\ne \es$ for all $1\le i\le m$, we say that $V_1,\ldots,V_m$ is a \defn{partition} of $V$.

Let $G=(V,E)$ be a graph. For $A,B\subseteq V$, let $G[A,B]$ be  
the subgraph of $G$ on vertex set $A\cup B$ with edges
\[
    \left\{ e\in E:\, e\subseteq A\cup B,\, e\cap A\ne \es,\, e\cap B\ne \es\right\}.
\]
In particular, for $A=B$ we have $G[A,A]=G[A]$, the induced subgraph of $G$ on $A$.

Let $V_1,\ldots,V_d$ be a partition of $V$. We say that $(V_1,\ldots,V_d)$ is a \defn{strong $d$-rigid partition} of $G$ if $G[V_i,V_j]$ is connected for all $1\le i\le j\le d$.

The following sufficient condition for rigidity was proved by Lew, Nevo, Peled, and Raz in~\cite{LNPR23+}%
\footnote{%
In~\cite{KLM23+},
we proved an extension of \cref{thm:strong_rigid_partition}
that applies to a more general notion
of ``rigid partitions''.}.

\begin{theorem}[{\cite[Theorem 1.3]{LNPR23+}}]\label{thm:strong_rigid_partition}
    Let $G=(V,E)$ be a graph with $|V|\ge d+1$. If $G$ admits a strong $d$-rigid partition, then $G$ is $d$-rigid.
\end{theorem}

\subsection{Concentration inequalities}
To conclude the preliminaries,
we state two known concentration inequalities that will be used in this paper.
The first is a standard Chernoff-type bound
(see, e.g.,~\cite{JLR}*{Theorem~2.1}).

\begin{theorem}[Chernoff bound]\label{thm:chernoff}
  Let $n\ge 1$ be an integer and let $p\in[0,1]$,
  let $X\sim\Bin(n,p)$
  and let $\mu=\E{X}=np$.
  Then, for every $\nu>0$,
  \[
    \pr(X\le \mu-\nu)
    \le \exp\left(-\frac{\nu^2}{2\mu}\right).
  \]
\end{theorem}

The second concentration inequality that we will need is the following consequence of Azuma's inequality
(see, e.g.,~\cite{JLR}*{Theorem~2.1}),
which can be easily proved along the lines of \cite{JLR}*{Corollary~2.27}
or \cite{frieze2004perfect}*{Lemma~11}.

For two sets $A,B$, we denote by $A\triangle B$ their symmetric difference.

\begin{lemma}\label{lem:azuma:worepl}
  Let $d>0$ and let $t>0$ be an integer.
  Let $V$ be a set with $|V|\ge t$. Suppose $f:\binom{V}{t}\to\RR$ satisfies the following Lipschitz condition:
  if $|A\triangle B|=2$,  then $|f(A)-f(B)|\le d$.
  Let $X=f(U)$ for a uniformly chosen $U\in\binom{V}{t}$.
  Then,
  \[
    \pr(|X-\E{X}|\ge \nu) \le 2\exp\left(-\frac{\nu^2}{2td^2}\right).
  \]
\end{lemma}

\section{Small $d$}\label{sec:small}
In this section we prove \cref{thm:exact}.
To this end, we begin with a few definitions.
A vertex is called \defn{simplicial} if its neighbourhood induces a clique.
The following lemma is implicit in \cite{Vil23+} (see also \cite{peled2024rigidity} for an explicit statement).
For completeness, we prove it here.
\begin{lemma}\label{lem:simpvx}
  Let $G=(V,E)$ be a $d$-closed graph with minimum degree at least $d(d+1)$.
  Then, $G$ contains a simplicial vertex.
\end{lemma}

For the proof of \cref{lem:simpvx} we will need the following two lemmas from \cite{Vil23+}.
Given a graph $G=([n],E)$ and a permutation $\sigma:[n]\to[n]$,
define the subgraph $G_\sigma$ of $G$ as follows.
For a vertex $u$, let $N_\sigma(u)=N(u)\cap\{v\in[n] :\sigma(v)<\sigma(u)\}$,
and let $\deg_\sigma(u)=|N_\sigma(u)|$.
Now, construct $E_\sigma\subseteq E$ by considering the vertices according to the order induced by $\sigma$.
For each vertex $v$,
\begin{itemize}
  \item If $\deg_\sigma(v)\le d$ then connect $v$ with $N_\sigma(v)$.
  \item If $\deg_\sigma(v)\ge d+1$ and $N_\sigma(v)$ is a clique,
        then connect $v$ with the first $d$ vertices in $N_\sigma(v)$.
  \item If $\deg_\sigma(v)\ge d+1$ and $N_\sigma(v)$ is not a clique,
        then connect $v$ with any $d+1$ vertices in $N_\sigma(v)$
        that do not form a clique in $G$.
\end{itemize}
Finally, define $G_\sigma=(V,E_\sigma)$.

\begin{lemma}[\cite{Vil23+}]\label{lem:Gsigma}
  Let $G=([n],E)$ be a $d$-closed graph and let $\sigma:[n]\to[n]$ be a permutation.
  Then, the graph $G_\sigma$ has at most $dn-\binom{d+1}{2}$ edges.
\end{lemma}

\begin{lemma}[\cite{Vil23+}]\label{lem:dpuz}
  Let $G=([n],E)$ be a graph with $\delta(G)\ge d(d+1)$
  in which for every vertex $v$, the neighbourhood of $v$ does not span a clique,
  and every two distinct maximal cliques $H_1,H_2\subseteq N(v)$ intersect in at most $d-2$ vertices.
  Then, there exists a permutation $\sigma:[n]\to[n]$ for which $G_\sigma$ has at least $dn$ edges.
\end{lemma}

We are now ready to prove \cref{lem:simpvx}.
\begin{proof}[Proof of \cref{lem:simpvx}]
  Assume, for contradiction, that $G$ has no simplicial vertex.
  Let $v\in V$, and let $H_1,H_2$ be two distinct maximal cliques in $N(v)$.
  In particular, $G[H_1\cup H_2\cup\{v\}]$ is not complete,
  and hence not $d$-rigid.
  Thus, $|(H_1\cup\{v\})\cap (H_2\cup\{v\})|\le d-1$
  (since $\OK_{|H_1|,|H_2|,d}$ is $d$-rigid).
  It follows that $|H_1\cap H_2|\le d-2$.
  Hence, by \cref{lem:dpuz,lem:Gsigma},
  there exists a permutation $\sigma:[n]\to[n]$ for which
  $dn\le|E_\sigma|\le dn-\binom{d+1}{2}$,
  a contradiction.
\end{proof}

We will also use the following simple lemma.

\begin{lemma}\label{lem:largerigid}
  Let $G$ be a $d$-closed graph on $n$ vertices, with minimum degree at least $\delta$. Assume that $G$ has a clique of size at least  $d+(n-1-\delta)$.
  Then, $G$ is a complete graph.
\end{lemma}

\begin{proof}
  Let $S$ be a maximal clique in $G$. In particular, $|S|\ge d+(n-1-\delta)$.
  Assume for contradiction that $S\ne V$, and let $S'=V\sm S$,
  so $|S'|=n-|S|\le \delta-d+1$.
  Let $v\in S'$, and note that $v$ has at least $\delta-(|S'|-1)\ge d$ neighbours in $S$.
  By \cref{lem:extension_in_closure}, $v$ is adjacent to all the vertices in $S$. That is, $S\cup\{v\}$ is a clique in $G$, in contradiction to the maximality of $S$.
\end{proof}

We proceed with the proof of \cref{thm:exact}.
\begin{proof}[Proof of \cref{thm:exact}]
  First note that by our bound on $d$,
  \[
    d(d+1)
    \le \frac{n}{2}+\frac{1}{8}\left(\sqrt{8n-15}-9\right)
    = \frac{n+d}{2}-1
    \le \delta(G).
  \]
  By \cref{lem:simpvx}, $G'=\Cl_d(G)$ has a simplicial vertex $v$.
  Hence, $S=N(v)\cup\{v\}$ induces a clique of size at least $\delta(G)+1\ge(n+d)/2 \ge d+(n-1-\delta(G))$ in $G'$. Therefore, by \cref{lem:largerigid}, $G'$ is a complete graph. Thus, $G$ is $d$-rigid.
\end{proof}

\section{Large $d$}\label{sec:large}
In this section we prove \cref{thm:approx}.
We do so by considering two separate cases,
according to how ``close'' the graph is to being bipartite.
More formally, we 
say that an $n$-vertex graph $G$ is \defn{$\beta$-far from being bipartite}
if one needs to delete at least $\beta n^2$ edges from $G$ to make it bipartite.
We call it \defn{$\beta$-close to being bipartite} otherwise.
The statement of \cref{thm:approx} follows
from \cref{thm:strong_rigid_partition} and
from \cref{prop:approx:close,prop:approx:far} below.
\begin{proposition}\label{prop:approx:close}
  There exist $\beta>0$, $c>0$, and $n_0\ge 2$ such that the following holds.
  For every $n\ge n_0$ and $1\le d\le cn/{\log^2}n$,
  every $n$-vertex graph $G$ with $\delta(G)\ge (n+2d)/2-1$
  which is $\beta$-close to being bipartite
  admits a strong $d$-rigid partition.
\end{proposition}
\begin{proposition}\label{prop:approx:far}
  For every $\beta>0$
  there exist $c>0$ and $n_0\ge 2$ such that the following holds.
  For every $n\ge n_0$ and $1\le d\le cn/\log{n}$,
  every $n$-vertex graph $G$ with $\delta(G)\ge (n+2d)/2-1$
  which is $\beta$-far from being bipartite
  is $d$-rigid.
\end{proposition}

Let us highlight that \cref{prop:approx:far}
applies for a wider range of $d$, compared with \cref{prop:approx:close}. On the other hand, note that \cref{prop:approx:close} has a somewhat stronger conclusion (the existence of a strong $d$-rigid partition, which implies the $d$-rigidity of the graph, by \cref{thm:strong_rigid_partition}).

\subsection{Close to bipartite}
In this section we prove \cref{prop:approx:close}.
We will need the following key lemma.

\begin{lemma}\label{lem:randomcol}
  There exists $c>0$ for which the following holds.
  Let $G=(V,E)$ be an $n$-vertex graph with $\delta(G)\ge d$ for $1\le d\le cn/{\log^2}{n}$.
  Then, there exists a distribution over $(d+1)$-colourings $V_1,\dots,V_{d+1}$ of $V$
  that satisfies the following properties:
  \begin{itemize}
    \item There exists a set $V'\subseteq V$
      with $|V'|\ge |V|-d$
      such that for every $v\in V'$ and $i\in[d]$,
      $\pr(v\in V_i)\ge 1/(2d)$,
      and the events $v\in V_i$ are mutually independent for distinct $v\in V'$;
    \item With high probability, $E(V_i,V_j)\ne\es$ for every $1\le i<j\le d+1$
      (namely, the colouring is \whp{} pseudocomplete).
  \end{itemize}
\end{lemma}

We remind the reader that, here and later, a colouring needs not be ``proper''.
Note that the colouring $V_1,\ldots,V_{d+1}$ guaranteed by \cref{lem:randomcol}
is (with high probability) a pseudocomplete colouring of $G$.
Therefore, \cref{thm:pseudo} follows directly from \cref{lem:randomcol}.
For the proof of \cref{prop:approx:close},
we will make use of the 
fact that \cref{lem:randomcol} allows for a substantial amount of randomness in the creation of the pseudocomplete colouring $V_1,\ldots,V_{d+1}$ of $G$.
 
We proceed to prove \cref{lem:randomcol}. We will need the following result
about expansion of random sets in graphs with relatively comparable degrees.

\begin{lemma}\label{lem:random_uniform_subset}
    Let $n,K,d\ge 1$ be integers satisfying $d\le n/(4K)$,
    and let $G=(V,E)$ be an $n$-vertex graph with $d\le \delta(G)\le \Delta(G)\le n/(4K)$.
    Then,
    if $U$ is a uniformly chosen random subset of $V$ of size $K$,
    the probability that $|N_G(U)|\ge d K/3$
    is at least $1-2\exp(-K/288)$.
\end{lemma}

\begin{proof}
  Write $\Delta=\Delta(G)$.
  For every $v\in V$ let $F_v$ be an arbitrary set of $d$ edges incident to $v$,
  and consider each such edge as directed from $v$.
  Let $D$ be the digraph on $V$ whose edge set is $F=\bigcup_v F_v$.
  Denote by $T$ the set of triples $(u_1,u_2,v)$ for which $u_1\ne u_2$ and $(u_1,v),(u_2,v)\in F$.
  Note that $|F|=dn$
  and that $|T|\le \Delta dn\le dn^2/(4K)$.
  For $A\subseteq V$, let $N_D(A)$ be the set of out-neighbours of $A$ in $D$,
  namely, the set of vertices $v\notin A$ such that there exists $u\in A$ for which $(u,v)\in F$.  
  Note that for every $A\subseteq V$, $N_D(A)\subseteq N_G(A)$.

  Let $f: \binom{V}{K}\to \RR$ be defined by $f(A)=|N_D(A)|$.
  Note that $|f(A)-f(B)|\le d+1$ whenever $|A\triangle B|=2$
  (since adding a new vertex to a set $A$ may increase $|N_D(A)|$ by at most $d$,
  or decrease it by at most $1$;
  similarly,
  removing a vertex from $A$ may increase $|N_D(A)|$ by at most $1$,
  or decrease it by at most $d$). 
  Let $U$ be a uniformly chosen random subset of $V$ of size $K$, and let $X=f(U)=|N_D(U)|$.
  
  Note that the events $u\in U$ are negatively correlated for distinct $u\in V$.
  Note also that $K/n\le 1/4$.
  Thus,
  \[\begin{aligned}
    \E{X}
    &= \E{\left[|N_D(U)|\right]}
    = \sum_v \pr(v\in N_D(U))\\
    &\ge \sum_v \left(
        \sum_{\substack{u\in V:\\(u,v)\in F}} \pr(u\in U,\ v\notin U)
        -\sum_{\substack{u_1,u_2\in V:\\(u_1,u_2,v)\in T}} \pr(u_1,u_2\in U,\ v\notin U)
      \right)\\
    &\ge \sum_{(u,v)\in F} \pr(u\in U,\ v\notin U)
         -\sum_{(u_1,u_2,v)\in T} \pr(u_1,u_2\in U)\\
    &\ge |F|\cdot \frac{K}{n}\left(1-\frac{K}{n}\right)
         -|T|\cdot \frac{K^2}{n^2}
      \ge dK(1-1/4)-dK/4 = dK/2.
  \end{aligned}\]
  Thus, by \cref{lem:azuma:worepl},
  \[\begin{aligned}
    \pr(|N_G(U)|\le d K/3)
    \le \pr(|N_D(U)|\le d K/3)
    &= \pr(X\le dK/3)
    \le\pr(X\le\E{X}-dK/6)\\
    &\le 2\exp\left(\frac{-d^2K^2}{72 K(d+1)^2}\right)
    \le 2\exp(-K/288),
  \end{aligned}\]
  as required.
\end{proof}

\begin{proof}[Proof of \cref{lem:randomcol}]
  Throughout the proof we may (and will) assume that $n$ is sufficiently large.
  Let $c>0$ to be sufficiently small constant (to be chosen later), and let $1\le d\le cn/{\log^2}{n}$.
  Let 
  \[
    L = \{v\in V:\, \deg(v)\ge 5 d \log{n}\},
    \qquad
    \ell=\min\{|L|,d\},
    \qquad
    d'=d-\ell,
  \]
  and let $L'=\{v_1,\dots,v_\ell\}$ be a set of $\ell$ arbitrary elements from $L$.
  Define a probability vector $\vect{q}=(q_1,\dots,q_{d+1})\in\RR^{d+1}$ as follows:
  if $\ell=0$ or $\ell= d$, then $q_i=1/(d+1)$ for all $1\le i\le d+1$;
  otherwise, $q_i=1/(2\ell)$ for $1\le i\le \ell$
  and $q_i=1/(2(d+1-\ell))$ for $\ell<i\le d+1$.
  Note that $q_i\ge 0$ for all $1\le i\le d+1$,
  and $\sum_{i=1}^{d+1}q_i=1$, so $\vect{q}$ is indeed a probability vector.
  We will need the following simple facts about $\vect{q}$.

  \begin{claim}\label{claim:vector_q} The vector $\vect{q}$ satisfies the following properties.
    \begin{enumerate}
      \item For all $1\le i\le d+1$, $q_i\ge 1/(2d)$.
      \item If $\ell<d$, then for all $\ell<i\le d+1$, $q_i\ge 1/(4d')$.
    \end{enumerate}
  \end{claim}
  \begin{proof}
    If $\ell=0$ or $\ell= d$, then
    $q_i=1/(d+1)\ge 1/(2d)$.
    Assume $1\le\ell<d$. Then, if $1\leq i\leq \ell$, 
    $q_i=1/(2\ell)>1/(2d)$,
    and if $\ell<i\leq d+1$,
    $q_i=1/(2(d+1-\ell))\ge 1/(2d)$.

    Now, assume $\ell<d$ and let $\ell<i\le d+1$.
    If $\ell=0$, then $q_i=1/(d+1)=1/(d'+1)$.
    If $0<\ell<d$, then $q_i=1/(2(d'+1))$.
    In both cases, $q_i \ge 1/(4d')$.
  \end{proof}

  Let $V'=V\sm L'$.
  Assign colours to the vertices of $V'$ independently, according to $\mathbf{q}$
  (that is, we have $\pr(v\in V_i)=q_i$ for all $v\in V'$ and $1\le i\le d+1$).
  In addition, assign $v_i\in V_i$ (deterministically) for all $1\le i\le \ell$.
  Observe that in view of \cref{claim:vector_q}, we have settled the first item of the lemma.

For $1\le i\le d$, let $\mathcal{B}_{i}$ be the event that $E(V_i,V_j)=\es$ for some $j>i$.
We want to show that $\pr(\bigcup_{i=1}^d \mathcal{B}_i)=o(1)$.
We begin by handling $\cB_i$ for $1\le i\le\ell$.
Let $1\le i\le \ell$.
Denote by $N_i$ the set of neighbours of $v_i$ in $V'$.
Note that $|N_i|\ge  5 d \log{n}- \ell\ge 4 d\log{n}$.
For every $j>i$,
by \cref{claim:vector_q},
$q_j\ge 1/(2d)$,
and hence $q_j |N_i|\ge 2\log{n}$.
Therefore, by the union bound over $j>i$, 
  \[
    \pr(\cB_i)
      \le \sum_{j>i}\pr(N_i\cap V_j=\es)
      \le \sum_{j>i}\left(1-q_j\right)^{|N_i|}
      \le \sum_{j>i}\exp(-q_j |N_i|) 
       \le d \cdot n^{-2} \ll n^{-1}.
  \]
By the union bound, we obtain
that the probability that $\cB_i$ occurs for some $1\le i\le \ell$ is $o(1)$.
If $\ell=d$, we are done. Otherwise, $d'\ge 1$. 

Write $n'=|V'|=n-\ell\sim n$.
Set $C=600$, $K=C\log{n}$
and $c=1/(40 C)$.
Fix $i<j\le d+1$ and let $x=|V_i|$ and $y=|V_j|$.
Notice that by \cref{claim:vector_q},
by Chernoff bounds (\cref{thm:chernoff}),
and by the union bound,
$\pr(\min\{x,y\}<n/(5d'))\le 2\exp(-n/(200d'))=o(n^{-2})$.
Note for later that $n/(5d')>K$ for sufficiently large $n$.
Observe that, conditioning on $x,y$, we may sample $V_i,V_j$ from $V'$
as follows:
first we sample a uniform subset $U$ of $V'$ of size $K'=\min\{K,x\}$;
then we sample a uniform subset $V_j$ of $V'\sm U$ of size $y$;
and finally we sample a uniform subset $V_i'$ of $V'\sm(U\cup V_j)$ of size $x-K'$, and set $V_i=U\cup V_i'$.
Note also that $d'\le \delta(G')\le \Delta(G')\le 5cn/\log{n}\le n'/(4K)$,
hence, by \cref{lem:random_uniform_subset},
for every $x_0\ge K$,
$\pr(|N_{G'}(U)|< d'K/3\mid x=x_0)\le 2\exp(-K/288) = o(n^{-2})$.
On the other hand, 
for every set $S\subseteq V'\setminus U$, every $x_0\ge K$ and every $y_0$,
\[
  \pr(S\cap V_j=\es\mid x=x_0,\ y=y_0)
  = \frac{\binom{n'-K-|S|}{y_0}}{\binom{n'-K}{y_0}}
  \le \left(1-\frac{|S|}{n}\right)^{y_0}
  \le \exp\left(-\frac{|S|y_0}{n}\right).
\]
Thus, for $x_0\ge K$ and $y_0\ge n/(5d')$,
\[
  \pr(N_{G'}(U)\cap V_j=\es\mid x=x_0,\ y=y_0)
  \le \pr\left(|N_{G'}(U)|\le d'K/3\mid x=x_0\right)
  + \exp(-K/15) = o\left(n^{-2}\right).
\]
Hence,
\begin{align*}
    &\pr(N_{G'}(U)\cap V_j=\es)
    = \sum_{x_0,y_0}
        \pr(N_{G'}(U)\cap V_j=\es\mid x=x_0,\ y=y_0)
        \cdot \pr(x=x_0,\ y=y_0)\\
    &\le \pr(\min\{x,y\}<n/(5d'))
       +\sum_{\substack{x_0\ge K\\y_0\ge n/(5d')}}
        \pr(N_{G'}(U)\cap V_j=\es\mid x=x_0,\ y=y_0)
        \cdot \pr(x=x_0,\ y=y_0)\\
    & = o\left(n^{-2}\right).
\end{align*}
Therefore, by the union bound over $j>i$, 
  \[
    \pr(\cB_i)
      \le \sum_{j>i}\pr(N_{G'}(U)\cap V_j=\es)
      = o\left(n^{-1}\right).
  \]
Finally, by another application of the union bound, we obtain
that the probability that $\cB_i$ occurs for some $\ell<i\le d$ is $o(1)$.
\end{proof}

\begin{proof}[Proof of \cref{prop:approx:close}]
For $d=1$ the claim follows from the known fact that an $n$-vertex graph with minimum degree at least $n/2-1/2$ is connected. Assume that $2\le d\le c'n/{\log^2}{n}$ for sufficiently small constant $c'>0$ to be determined later,
  and that $n$ is sufficiently large.
  By our assumption that $G=(V,E)$ is $\beta$-close to being bipartite, 
  we deduce that there exists a partition $V=A\cup B$
  for which $|E(A)|+|E(B)|\le\beta n^2$.
  Let $\beta'=4\beta$.
  Assume for contradiction that $|A|<(1/2-\beta')n$.
  In this case, $|B|>(1/2+\beta')n$.
  Since $\delta(G)\ge (n+2d)/2-1 \ge n/2$, it follows that every $v\in B$ has at least $\deg(v)-|A|\ge \beta'n$ neighbours in $B$.
  In particular, $|E(B)|\ge\frac{1}{2}|B|\cdot \beta'n>\beta n^2$,
  a contradiction.
  Therefore, $|A|\ge(1/2-\beta')n$. By the same argument, we also have $|B|\ge(1/2-\beta')n$.

  Let $\beta^\circ=\sqrt{2\beta}$.
  Let $A^\circ=\{u\in A\mid \deg(u,B)<(1/2-\beta^\circ)n\}$.
  Suppose to the contrary that $|A^\circ|>\beta^\circ n$.
  Thus, $|E(A)|>\frac{1}{2}\beta^\circ n\cdot(\delta(G)-(1/2-\beta^\circ)n)
  \ge \frac{1}{2}(\beta^\circ)^2n^2=\beta n^2$,
  a contradiction.
  Therefore, $|A^\circ|\le\beta^\circ n$.
  Similarly, letting 
  $B^\circ=\{v\in B\mid \deg(v,A)<(1/2-\beta^\circ)n\}$,
  we deduce that $|B^\circ|\le\beta^\circ n$.

  We will now make some adjudgements to the partition $(A,B)$.
  Let $A^\dagger=\{u\in A\mid \deg(u,B)<n/4\}$,
  $B^\dagger=\{v\in B\mid \deg(v,A)<n/4\}$,
  and $C=A^\dagger\cup B^\dagger$.
  As long as $C\ne\es$ we take $x\in C$ and move it to the other side
  (namely, if $x\in A$ we put it in $B$, and if $x\in B$ we put it in $A$).
  We note that once we move $x$ to the other side, it is not in $C$ any longer.
  However, this might introduce a new vertex to $C$.
  Nonetheless, we argue that the total number of vertices we move is at most $2\beta^\circ n$.
  We prove this by showing that every vertex we moved was in $A^\circ\cup B^\circ$.
  Indeed, at first every vertex of $C$ is there (by definition, assuming sufficiently small $\beta$).
  Suppose to the contrary that at some point we have a vertex in $C\sm (A^\circ\cup B^\circ)$,
  and let $x$ be the first such vertex.
  Since $x$ is first, by this time we have moved at most $2\beta^\circ n$ vertices.
  But then, the degree of $x$ to the other side is at least $(1/2-3\beta^\circ)n>n/4$
  (for sufficiently small $\beta$),
  a contradiction.
  Let $(A',B')$ denote the resulting partition,
  let $\beta_0=\beta+2\beta^\circ$,
  and let $\beta_1=\beta'+2\beta^\circ$.
  Since we moved at most $2\beta^\circ$ vertices,
  we have $|A'|,|B'|\ge (1/2-\beta_1)n$.
  Since each moved vertex adds to $|E(A)\cup E(B)|$ at most $n$ edges,
  we have $|E(A')|+|E(B')|\le \beta_0n^2$.
    Also, by definition, $\deg(v,B')\geq n/4$ for all $v\in A'$ and $\deg(v,A')\geq n/4$ for all $v\in B'$.
  
  Let $\beta^*=\sqrt{2\beta_0}$,
  and let $A^*=\{u\in A'\mid \deg(u,B')<(1/2-\beta^*)n\}$.
  Suppose to the contrary that $|A^*|>\beta^* n$.
  Then, $|E(A')|>\frac{1}{2}\beta^* n\cdot(\delta(G)-(1/2-\beta^*)n)
  \ge \frac{1}{2}(\beta^*)^2n^2=\beta_0 n^2$,
  a contradiction.
  Therefore, $|A^*|\le\beta^* n$.
  Similarly, letting 
  $B^*=\{v\in B'\mid \deg(v,A')<(1/2-\beta^*)n\}$,
  we deduce that $|B^*|\le\beta^* n$.

  Assume without loss of generality that $|A'|\ge |B'|$, and hence $|B'|\le n/2$
  and $n':=|A'|\ge n/2$.
  Then, for all $v\in A'$, $\deg(v,A')\ge\delta(G)-|B'|\ge n/2-1+d-n/2=d-1$.
  Thus, $\delta(G[A'])\ge d-1$. 
  Let $c>0$ be the constant guaranteed by \cref{lem:randomcol}, and set $c'=c/2$.
  Then, $d-1\le c'n/{\log^2}{n} \le cn'/{\log^2}{n'}$.
  Let $A'=A_1\cup\dots\cup A_d$ be the random colouring guaranteed in \cref{lem:randomcol},
  and let $B'=B_1\cup\dots\cup B_d$ be a uniform random colouring.

    \begin{claim}\label{claim:close1}
        With high probability, for all $1\leq i\leq d$ and $b_1,b_2\in B'\sm B^{*}$, $b_1$ and $b_2$ have a common neighbour in $A_i$.
    \end{claim}
    \begin{proof}
    Let $1\leq i\leq d$ and $b_1,b_2\in B'\sm B^*$. 
    Note that $|N(b_1)\cap N(b_2)\cap A'|\ge (1/2-\beta^*)n+(1/2-\beta^*)n-(1/2+\beta_1)n\ge   n/3$ (for sufficiently small $\beta$). Hence, using the first item of \cref{lem:randomcol},
  \[
  \pr(N(b_1)\cap N(b_2)\cap A_i=\es) \le \left(1-\frac{1}{2d}\right)^{n/3-(d-1)}\le \left(1-\frac{1}{2d}\right)^{n/4} \le \exp\left(-\frac{n}{8d}\right) = o(n^{-3}).
  \]
  By the union bound, with probability at least $1-n^2 d \cdot o(n^{-3}) = 1-o(1)$, for all $1\leq i\leq d$ and $b_1,b_2\in B'\sm B^{*}$, $b_1$ and $b_2$ have a common neighbour in $A_i$.
    \end{proof}

\begin{claim}\label{claim:close2}
With high probability, every $a\in A'$ has a neighbour in $B_j\sm B^{*}$ for all $1\leq j\leq d$, and every $b\in B'$ has a neighbour in $A_i$ for all $1\leq i\leq d$.
\end{claim}
\begin{proof}
    Let $1\leq j\leq d$ and $a\in A'$. Then, $\deg(a,B'\sm B^*)\geq n/4-\beta^{*}n \geq n/5$ (for sufficiently small $\beta$). Therefore,
    \[
        \pr(N(a)\cap B_j\sm B^{*} = \es)\le  \left(1-\frac{1}{d}\right)^{n/5}
    \le \exp\left(-\frac{n}{5d}\right)
    = o\left(n^{-2}\right).
    \]
    By the union bound, with probability at least $1-nd\cdot o(n^{-2})=1-o(1)$, every $a\in A'$ has a neighbour in $B_j\sm B^*$ for all $1\leq j\leq d$.

    Similarly, let $1\le i\le d$ and $b\in B'$. Then $\deg(b,A')\geq n/4$, so by the first item of \cref{lem:randomcol}, we have
    \[
        \pr(N(b)\cap A_i=\es)\le \left(1-\frac{1}{2d}\right)^{n/4-(d-1)}\le \left(1-\frac{1}{2d}\right)^{n/5} \le \exp\left(-\frac{n}{10d}\right) = o(n^{-2}). 
    \]
    By the union bound, with probability at least $1-nd\cdot o(n^{-2})=1-o(1)$, every $b\in B'$ has a neighbour in $A_i$ for all $1\leq i\leq d$.
\end{proof}

Condition on the events from \cref{lem:randomcol,claim:close1,claim:close2}.
Let $1\leq i,  j\leq d$.  Then, by \cref{claim:close2}, every vertex in $A_i\cup B_j$ has a path of length at most two in $G[A_i,B_j]$ to a vertex in $B_j\sm B^{*}$. By \cref{claim:close1}, every two vertices in $B_j\sm B^{*}$ are connected by a path of length two in $G[A_i,B_j]$. Therefore, $G[A_i,B_j]$ is connected.

 Define, for all $1\le i\le d$, $V_i=A_i\cup B_i$.
  We note that $G[V_i]\supseteq G[A_i,B_i]$, hence $G[V_i]$ is connected.
  Moreover, for $1\le i<j\le d$,
  the sets $A_i,B_j$ are in the same component in $G[V_i,V_j]$.
  Similarly,
  the sets $A_j,B_i$ are in the same component there.
  But, by \cref{lem:randomcol}, $E(A_i,A_j)\ne\es$;
  thus, there exists a unique connected component in $G[V_i,V_j]$,
  namely, it is connected.
  Thus, $(V_1,\dots,V_d)$ is a strong $d$-rigid partition of $G$. 
\end{proof}

\subsection{Far from bipartite}
In this section we prove \cref{prop:approx:far}.
As our main tool, we will use the regularity lemma.
Let us begin by recalling relevant definitions (see \cite{KS96}). 
For a disjoint pair of vertex sets $X,Y$ in a graph $G$,
its \defn{density} is $d(X,Y)=|E(X,Y)|/(|X||Y|)$.
A pair $(A,B)$ is called:
\begin{itemize}
  \item 
    \defn{$\eps$-regular}
    if for all $X\subseteq A$ and $Y\subseteq B$
    with $|X|\ge\eps|A|$ and $|Y|\ge\eps|B|$
    we have $|d(X,Y)-d(A,B)|\le\eps$;
  \item
    \defn{$(\eps,\delta)$-regular}
    if it is $\eps$-regular with $d(A,B)\ge \delta$;
  \item
    \defn{$(\eps,\delta)$-dense}
    if $d(X,Y)\ge\delta$ for every $X\subseteq A$ and $Y\subseteq B$
    with $|X|\ge\eps|A|$ and $|Y|\ge\eps|B|$;
  \item
    \defn{$(\eps,\delta)$-super-regular}
    if it is $\eps$-regular and, in addition,
    each $x\in A$ has at least $\delta|B|$ neighbours in $B$
    and each $y\in B$ has at least $\delta|A|$ neighbours in $A$.
\end{itemize}
Note that:
\begin{itemize}
  \item
    If $(A,B)$ is $(\eps,\delta)$-regular
    then for every $X\subseteq A$ and $Y\subseteq B$
    with $|X|\ge\eps|A|$ and $|Y|\ge\eps |B|$ we have $d(X,Y)\ge\delta-\eps$,
    that is, $(A,B)$ is $(\eps,\delta-\eps)$-dense.
  \item 
    If $(A,B)$ is $(\eps,\delta)$-super-regular
    then $d(A,B)\ge\delta$,
    that is, $(A,B)$ is $(\eps,\delta)$-regular
    (and hence $(\eps,\delta-\eps)$-dense).
\end{itemize}

We will need the following standard and simple lemma that asserts that every regular ``triple''
contains an almost-spanning super-regular ``sub-triple''.
\begin{lemma}\label{lem:superreg:3}
  Let $0<\eps<1/4$ and $\delta>4\eps$.
  Let $A_1,A_2,A_3$ be three pairwise disjoint sets with $|A_1|=|A_2|=|A_3|=M$
  such that each of the pairs $(A_i,A_j)$, $1\le i<j\le 3$,
  is a $(\eps,\delta)$-regular pair.
  Then, there exist $M'\ge(1-2\eps)M$
  and
  $A_i'\subseteq A_i$ with $|A'_i|=M'$, for every $i\in[3]$,
  such that each of the pairs $(A_i',A_j')$, $1\le i<j\le 3$,
  is $(2\eps,\delta-4\eps)$-super-regular.
\end{lemma}

\begin{proof}
  For $i\in[3]$ and $j\in[3]\sm\{i\}$, let $A_i^j=\{x\in A_i\mid \deg(x,A_j)<(\delta-2\eps)|A_j|\}$.
  Note that $|A_i^j|<\eps|A_i|$.
  Indeed,
  if $|A_i^j|\ge\eps |A|$ then, since $(A_i,A_j)$ is $(\eps,\delta)$-regular,
  \[
    |A_i^j|\cdot(\delta-2\eps)|A_j|>|E(A_i^j,A_j)|
    =|A_i^j||A_j|\cdot d(A_i^j,A_j)\ge|A_i^j||A_j|(\delta-\eps),
  \]
  a contradiction.
  For $i\in[3]$ and $\{j,k\}=[3]\sm\{i\}$,
  Let $A_i^*=A_i\sm (A_i^j\cup A_i^{k})$.
  Note that $|A_i^*|\ge(1-2\eps)|A_i|$.
  Set $M'=\min\{|A_i^*|:i\in[3]\}$, so $M'\ge(1-2\eps)M$.
  For each $i\in[3]$, remove, if needed, arbitrary elements from $A_i^*$,
  to make it of size $M'$, and call the resulting set $A_i'$.
  In particular, $|A_i'|\ge(1-2\eps)|A_i|\ge |A_i|/2\ge\eps|A_i|$.
  Thus, for $j\in[3]\sm\{i\}$,
  $|d(A_i',A_j')-d(A_i,A_j)|\le\eps$.
  Moreover,
  for every $X_i\subseteq A_i'$ and $X_j\subseteq A_j'$
  with $|X_i|\ge 2\eps|A_i'|\ge \eps|A_i|$ and $|X_j|\ge 2\eps|A_j'|\ge \eps|A_j|$,
  we have $|d(X_i,X_j)-d(A_i',A_j')|\le|d(X_i,X_j)-d(A_i,A_j)|+|d(A_i,A_j)-d(A_i',A_j')|\le 2\eps$.
  Hence, $(A_i',A_j')$ is $2\eps$-regular.
  In addition, if $x\in A_i'$, 
  then $\deg(x,A_j')\ge\deg(x,A_j)-2\eps|A_j|\ge(\delta-4\eps)|A_j|\ge(\delta-4\eps)|A_j'|$.
  Thus, $(A_i',A_j')$ is $(2\eps,\delta-4\eps)$-super-regular.
\end{proof}

We will use the following degree version of the regularity lemma (see~\cite{KS96}).

\begin{lemma}[Regularity lemma, degree version]\label{lem:reg}
  For every $\eps>0$ and $\ell_0>0$
  there exists an integer $L=L(\eps,\ell_0)$
  such that for every $n$-vertex graph $G=(V,E)$ with $n\ge L$
  and every $\delta\in[0,1]$
  there exists a partition of $V$ into $\ell+1$ sets $V_0,V_1,\dots,V_\ell$
  with $\ell_0\le\ell\le L$,
  and a spanning subgraph $G'$ of $G$ with the following properties:
  \begin{enumerate}
    \item $|V_0|\le\eps n$ and $|V_i|=M$ for all $i\in[\ell]$ and some $M>0$;
    \item $\deg_{G'}(v)\ge\deg_G(v)-(\delta+\eps)n$ for all $v\in V$;
    \item $V_i$ is independent in $G'$ for all $i\in[\ell]$;
    \item For every $1\le i<j\le \ell$,
      either $E_{G'}(V_i,V_j)=\es$
      or $(V_i,V_j)$ is $(\eps,\delta)$-regular in $G'$.
  \end{enumerate}
\end{lemma}

Given a graph $G=(V,E)$,
a partition $V=V_1\cup\dots\cup V_\ell$,
and parameters $\eps,\delta>0$,
let $R(G,(V_1,\dots,V_\ell),\eps,\delta)$
be the graph on the vertex set $\{V_1,\dots,V_\ell\}$
where $V_i\sim V_j$ if and only if $(V_i,V_j)$ is $(\eps,\delta)$-regular in $G$.
We think of the regularity lemma as an algorithm:
its input is $\eps,\ell_0,n,\delta$ with $\eps>0$, $\ell_0\ge 1$, $n\ge L(\eps,\ell_0)$,
$\delta\in[0,1]$, and an $n$-vertex graph $G$;
and its output is a partition $V_0,V_1,\dots,V_\ell$ with the desired properties,
the so-called \defn{pure graph} $G''=G'[V\sm V_0]$,
and the \defn{reduced graph} $R=R(G'',(V_1,\dots,V_\ell),\eps,\delta)$.

\begin{lemma}\label{lem:pure:size}
  Let $\eps,\ell_0,n,\delta$ with $\eps>0$, $\ell_0\ge 1$, $n\ge L(\eps,\ell_0)$, and $\delta\in[0,1]$,
  and let $G=(V,E)$ be an $n$-vertex graph.
  Let $G''$ be the pure graph obtained by \cref{lem:reg} with the above parameters.
  Then: $|E(G'')|\ge |E(G)| - (\delta+3\eps)n^2/2$.
\end{lemma}

\begin{proof}
  We recall that for every $v\in V$,
  $\deg_{G''}(v)\ge\deg_{G'}(v)-|V_0|\ge \deg_G(v)-(\delta+2\eps)n$.
  Thus,
  \[\begin{aligned}
    2|E(G'')|
      & =\sum_{v\in V\sm V_0} \deg_{G''}(v)
        \ge \sum_{v\in V} \deg_G(v)- \sum_{v\in V\sm V_0}(\delta+2\eps)n - \sum_{v\in V_0}\deg_G(v)\\
      &\ge 2|E(G)|-(\delta+2\eps)n^2-\eps n^2
       = 2|E(G)|-(\delta+3\eps)n^2,
  \end{aligned}\]
  as required.
\end{proof}

We deduce that the reduced graph has, approximately, the ``correct'' number of edges.
\begin{corollary}\label{cor:reduced:size}
  Let $\eps,\ell_0,n,\delta$ with $\eps>0$, $\ell_0\ge 1$, $n\ge L(\eps,\ell_0)$, and $\delta\in[0,1]$,
  and let $G=(V,E)$ be an $n$-vertex graph. Let $V_0,V_1,\dots,V_\ell$ be the partition obtained by \cref{lem:reg}
  with the above parameters, and let $R$ be the corresponding reduced graph. 
  Then: $|E(R)|\ge \frac{\ell^2}{n^2}\cdot|E(G)| - (\delta+3\eps)\ell^2/2$.
\end{corollary}

\begin{proof}
 Let $G''$ be the pure graph. Recall that by \cref{lem:reg}, $|V_i|=M$ for all $1\leq i\leq \ell$. Therefore, $M=|V_1|\le n/\ell$.
  By \cref{lem:pure:size},
  $|E(G'')|\ge |E(G)| - (\delta+3\eps)n^2/2$.
  On the other hand,
  $|E(G'')|\le M^2|E(R)|$,
  hence $|E(R)| \ge \frac{\ell^2}{n^2}\cdot|E(G)|-(\delta+3\eps)\ell^2/2$.
\end{proof}

We now show that if $G$ is far from being bipartite and dense enough, then its reduced graph contains a triangle.
The following lemma makes this statement precise:
\begin{lemma}\label{lem:far:triangle}
  Let $\beta>0$.
  Then, there exist sufficiently small $\eps>0$ and sufficiently large $\ell_0$ such that the following holds.
  Let  $n\ge L(\eps,\ell_0)$, let $\delta=8\eps$,
  and let $G=(V,E)$ be an $n$-vertex graph
  with $|E(G)|\ge n^2/4$,
  which is $\beta$-far from being bipartite.
  Let $R$ be the reduced graph obtained by \cref{lem:reg} with the above parameters.
  Then, $R$ contains a triangle.
\end{lemma}

For the proof of \cref{lem:far:triangle}
we will use the following special case of a structural stability result of Erd\H{o}s and Simonovitz
(\cite{ES96}; see also e.g. ~\cite[Theorem 1]{Fur15}).

\begin{theorem}\label{thm:stab}
  For every $\beta>0$ there exists $\eta>0$ and $n_0\ge 1$
  such that if $n\ge n_0$ and $G=(V,E)$ is an $n$-vertex triangle-free graph
  with $|E|\ge \left(\frac{1}{4}-\eta\right)n^2$
  then $G$ is $\beta$-close to being bipartite.
\end{theorem}

We will also need the following simple lemma, according to which if $R$ is close to being bipartite, then so is $G$.
\begin{lemma}\label{lem:close}
  Let $\eps,\ell_0,n,\delta$ with $\eps>0$, $\ell_0\ge 1$, $n\ge L(\eps,\ell_0)$, and $\delta\in[0,1]$,
  and let $G=(V,E)$ be an $n$-vertex graph.
  Let $R$ be the reduced graph obtained by \cref{lem:reg} with the above parameters.
  Let $\beta\ge (\delta+3\eps)/2$, and assume that $R$ is $(\beta-(\delta+3\eps)/2)$-close to being bipartite.
  Then, $G$ is $\beta$-close to being bipartite.
\end{lemma}

\begin{proof}
  We show that by removing $\beta n^2$ edges from $G$ we can make it bipartite.
  Indeed,
  write $\gamma=(\delta+3\eps)/2$
  and $\beta_R=\beta-\gamma$.
  Let $F$ be a set of edges of $R$ such that $|F|\le\beta_R\ell^2$ and $R-F$ is bipartite.
  Let $F'=\{\{u,v\}\in E\mid u\in V_i,\ v\in V_j,\ \{V_i,V_j\}\in F\}$,
  and note that $|F'|\le M^2|F|$ for $M=|V_1|\le n/\ell$, so $|F'|\leq \beta_R n^2$.
  Let $G''$ be the corresponding pure graph,
  and denote $F''=E(G)\sm E(G'')$.
  By \cref{lem:pure:size},
  we have that $|F''|\le \gamma n^2$.
  Set $F^\circ=F'\cup F''$.
  Thus,
  $|F^\circ|\le (\beta_R+\gamma)n^2=\beta n^2$.
  Finally, we observe that $G-F^\circ$ is bipartite,
  since any odd cycle in $G-F^\circ$ gives rise naturally to an odd cycle in $R-F$.
\end{proof}

\begin{proof}[Proof of \cref{lem:far:triangle}]
  Apply \cref{thm:stab} with the parameter $\beta'=\beta/2$
  to obtain $\eta,\ell_0$.
  Set $\eps=\min\{\eta/6,\beta/11\}$,
  and let $\gamma=(\delta+3\eps)/2=11\eps/2\le\beta'$.
  Suppose to the contrary that $R$ is triangle-free.
  By \cref{cor:reduced:size},
  $R$ has at least
  \[
    \left(\frac{1}{4}-(\delta+3\eps)/2\right)\ell^2
    \ge\left(\frac{1}{4}-6\eps\right)\ell^2
    \ge\left(\frac{1}{4}-\eta\right)\ell^2
  \]
  edges, where $\ell_0\leq \ell\leq L(\eps,\ell_0)$ is the number of vertices of $R$.
  By \cref{thm:stab}, 
  $R$ is $\beta'$-close to being bipartite.
  Noting that $\beta'=\beta/2 \le \beta-\gamma$,
  we conclude that $R$ is $(\beta-\gamma)$-close to being bipartite.
  By \cref{lem:close},
  $G$ is $\beta$-close to being bipartite, a contradiction.
\end{proof}

Putting it all together, we conclude that $G$ contains a triple of super-regular pairs.
\begin{corollary}\label{cor:supertriple}
  Let $\beta>0$.
  Then, there exist sufficiently small $\delta,\alpha>0$ and sufficiently large $n_0$
  such that the following holds.
  If $G=(V,E)$ is an $n$-vertex graph with $n\ge n_0$ and $|E(G)|\ge n^2/4$
  which is $\beta$-far from being bipartite,
  then there exist pairwise disjoint $A_1,A_2,A_3\subseteq V$
  with $|A_1|=|A_2|=|A_3|\ge \alpha n$
  and a subgraph $G''\subseteq G$,
  such that the pairs $(A_i,A_j)$, $1\le i<j\le 3$,
  are all $(\delta,2\delta)$-super-regular in $G''$.
\end{corollary}

\begin{proof}
  Let $\eps=\eps(\beta)>0$ be sufficiently small
  and $\ell_0=\ell_0(\beta)$ be sufficiently large
  so that the statement of \cref{lem:far:triangle} holds.
  Apply the regularity lemma with the parameters
  $\eps,\ell_0,n,\delta'$ with $n\ge L=L(\eps,\ell_0)$ and $\delta'=8\eps$.
  Let $G''$ and $R$ be the pure graph and the reduced graph, respectively.
  According to \cref{lem:far:triangle}, $R$ contains a triangle.
      Assume that the vertices of this triangle are $V_1,V_2,V_3$, and let $M=|V_1|\ge n/(L+1)$.
  Note that each of the pairs $(V_1,V_2)$, $(V_1,V_3)$ and $(V_2,V_3)$
  is $(\eps,\delta')$-regular in $G''$.
  We apply \cref{lem:superreg:3} to the triple $(V_1,V_2,V_3)$
  to obtain subsets $A_1\subseteq V_1$, $A_2\subseteq V_2$ and $A_3\subseteq V_3$
  with $|A_1|=|A_2|=|A_3|=M'\ge(1-2\eps)M$,
  such that the pairs $(A_1,A_2)$, $(A_1,A_3)$ and $(A_2,A_3)$
  are all $(2\eps,\delta'-4\eps)$-super-regular in $G''$.
  We finish by choosing $\delta=2\eps$
  and $\alpha=(1-2\eps)/(L+1)$.
\end{proof}

\subsubsection{Super-regular tripartite graphs}

Next, we show that a triple of disjoint $n$-vertex sets, such that each pair of sets forms a super-regular pair, admits a strong $d$-rigid partition for all $d=O(n/\log{n})$.

\begin{proposition}\label{prop:triangle}
  For every $\delta>0$
  there exists $c>0$
  such that the following holds.
  Let $n,d\ge 1$
  be such that $d\le cn/\log{n}$.
  Let $G=(V,E)$ be a tripartite graph with sides $A_1,A_2,A_3$,
  each of size $n$.
  Assume that for all $1\leq i<j\leq 3$, the pair $(A_i,A_j)$ is $(\delta,2\delta)$-super-regular in $G$.
  Then,
  $G$ admits a strong $d$-rigid partition.
\end{proposition}
\begin{proof}

Let $c>0$ be a constant to be chosen later,
and take $1\le d\le cn/\log{n}$.
 Let $V_1,\ldots,V_d$ be a uniform random $d$-colouring of $V$;
 namely,
 for every $v\in V$
 independently,
 $\pr(v\in V_i)=1/d$ for every $i\in[d]$.
 
\begin{claim}\label{claim:n1}
  With probability at least $1-6n^{2-2\delta/c}$,
  for all $i\in[3]$, $v\in A_i$, $j\in[3]\sm\{i\}$ and $1\leq k\leq d$,
  $v$ has neighbour in $V_k\cap A_j$.
\end{claim}

\begin{proof}
  For every $i\in[3]$, $v\in A_i$, $j\in[3]\sm\{i\}$ and $k\in [d]$,
  let $\cE_{v,j,k}$ be the event that $v$ does not have a neighbour in $A_j\cap V_k$.
  Note that by the $(\delta,2\delta)$-super-regularity of $(A_i,A_j)$,
  we have $\deg(v,A_j)\ge 2\delta n$.
  Thus,
  \[
    \pr(\cE_{v,j,k})\leq (1-1/d)^{2\delta n} \leq e^{-2\delta n/d} \le n^{-2\delta/c}.
  \]
  By the union bound
  over all $i\in[3]$,
  $v\in A_i$,
  $j=[3]\sm\{i\}$,
  and $k\in[d]$,
  the probability that
  $\cE_{v,j,k}$ occurs for
  some $v,j,k$
  is at most
  \[
      3\cdot n\cdot 2\cdot d \cdot n^{-2\delta/c} \leq 6 n^{2-2\delta/c}.
  \]
  Therefore, with probability at least $1-6n^{2-2\delta/c}$,
  for all $i\in[3]$, $v\in A_i$, $j\in[3]\sm\{i\}$ and $k\in[d]$,
  $v$ has neighbour in $A_j\cap V_k$.
\end{proof}

\begin{claim}\label{claim:n2}
  With probability at least $1-6 n^{4-2\delta^3/c}$,
  for all $i\in[3]$,
  $x\ne y$ in $A_i$,
  and $r,s\in[d]$,
  $E(N(x)\cap V_r,N(y)\cap V_s)\ne\es$.
\end{claim}

\begin{proof}
  For $i\in[3]$,
  $x\ne y$ in $A_i$,
  and $r,s\in[d]$, 
  let $\cE_{x,y,r,s}$ be the event that
  $E(N(x)\cap V_r,N(y)\cap V_s)=\es$.
  Fix such $i, x,y ,r,s$ and
  let $\{j,k\}=[3]\sm\{i\}$.
  By the super-regularity of $(A_i,A_j)$ in $G$,
  $|N(x)\cap A_j|\ge 2\delta n > \delta n$.
  By the super-regularity of $(A_i,A_k)$,
  $|N(y)\cap A_k|\ge 2\delta n > \delta n$.
  Hence, by the ($\delta$,$\delta$)-density of $(A_j,A_k)$
  (which follows from its $(\delta,2\delta)$-regularity; see notes after the corresponding definitions),
  $|E(N(x)\cap A_j,N(y)\cap A_k)|\ge\delta\cdot (2\delta n)^2= 4 \delta^3 n^2$.
  Let 
  \[
    L=\left\{v\in N(x)\cap A_j:\, \deg(v,N(y)\cap A_k)\geq 2\delta^3 n\right\}.
  \]
  Then, $|L|\geq 2\delta^3 n$
  (otherwise,
   $|E(N(x)\cap A_j,N(y)\cap A_k)|\le 2\delta^3 n \cdot n + (n-2\delta^3 n)\cdot 2\delta^3 n
    < 4\delta^3 n^2$,
   a contradiction). 
  Therefore,
  \[
    \pr(L\cap V_r=\es) \leq (1-1/d)^{2\delta^3 n} \leq e^{-2\delta^3 n/d} \le n^{-2\delta^3/c}.
  \]
  Conditioning on $L\cap V_r \neq \es$,
  there are at least $2\delta^3 n$ vertices in $N(y)\cap A_k$
  that are adjacent to at least one vertex in $N(x)\cap A_j\cap V_r$,
  and therefore,
  \begin{align*}
  & \pr(E(N(x)\cap A_j\cap V_r,N(y)\cap A_k\cap V_s)=\es  \,| \, L\cap V_r\ne \es) \\
    & \leq (1-1/d)^{2\delta^3 n} \leq n^{-2\delta^3/c}.
  \end{align*}
  Hence,
  \begin{align*}
    & \pr(\cE_{x,y,r,s})\leq 
    \pr(E(N(x)\cap A_j\cap V_r,N(y)\cap A_k\cap V_s)=\es)\leq 2 n^{-2\delta^3/c}.
  \end{align*}
  By the union bound
  over all $i\in[3]$,
  $x\ne y$ in $A_i$,
  and $r,s\in[d]$,
  the probability that
  $\cE_{x,y,r,s}$ occurs for
  some $x,y,r,s$
  is at most
  \[
    3n^2 d^2 \cdot 2n^{-2\delta^3/c}\leq 6n^{4-2\delta^3/c}.
  \]
  Thus, with probability at least $1-6 n^{4-2\delta^3/c}$,
  for all $i\in[3]$,
  $x\ne y$ in $A_i$,
  and $r,s\in[d]$,
  we have that $E(N(x)\cap V_r,N(y)\cap V_s)\ne\es$.
\end{proof}

  Let $c=2\delta^3/(5+\log_2(12))$.
  Let $\cE_1$ be the event that there exist $i\in[3]$, $v\in A_i$, $j\in[3]\sm\{i\}$, and $k\in[d]$
  such that $v$ does not have a neighbour in  $V_k\cap A_j$.
  By \cref{claim:n1}, $\pr(\cE_1)\le 6n^{2-2\delta/c}$.
  Let $\cE_2$ be the event that there exist $i\in[3]$, $x,y\in A_i$ such that $x\ne y$, and $r,s\in[d]$
  such that $E(N(x)\cap V_r,N(y)\cap V_s)=\es$.
  By \cref{claim:n2}, $\pr(\cE_2)\le 6n^{4-2\delta^3/c}$.
  Thus, by the union bound, noting that $4-2\delta^3/c > 2-2\delta/c$,
  \[
    \pr(\cE_1\cup \cE_2)\leq 6n^{2-2\delta/c}+6 n^{4-2\delta^3/c}
      \leq 6\left(2 ^{2-2\delta/c}+ 2^{4-2\delta^3/c}\right)
      \leq 12 \cdot 2^{4-2\delta^3/c}<1.
  \]
  Hence, there exists a partition $(V_1,\ldots,V_d)$ for which both $\cE_1$ and $\cE_2$ do not occur.
  We will show that $(V_1,\ldots,V_d)$ is a strong $d$-rigid partition of $G$.

  First, let $r\in[d]$.
  We will show that $G[V_r]$ is connected. Let $x,y\in V_r$ with $x\ne y$. We divide into two cases: 
  \begin{description}
    \item[Case 1.] 
      If $x,y\in A_i$ for some $i\in [3]$, then (since $\cE_2$ does not occur)
      there is some $u\in N(x)\cap V_r$ and $v\in N(y)\cap V_r$ such that $\{u,v\}\in E$.
      Therefore, $x$ and $y$ are connected  in $G[V_r]$ by a path of length three. 
    \item[Case 2.]
      If $x\in A_i$ and $y\in A_j$ for $i\neq j$,
      then (since $\cE_1$ does not occur) there exists $z\in N(y)\cap V_r \cap A_i$.
      By the previous case, there is a path in $G[V_r]$ between $x$ and $z$,
      and therefore there is a path in $G[V_r]$ between $x$ and $y$.
  \end{description}
  We conclude that $G[V_r]$ is connected.
  Now, let $r,s\in[d]$ with $r\ne s$.
  We will show that $G[V_r,V_s]$ is connected.
  Let $x,y\in V_r\cup V_s$ with $x\ne y$. We divide into four cases:
  \begin{description}
    \item[Case 1.] 
      Assume $x\in V_r\cap A_i$ and $y\in V_s\cap A_i$ for some $i\in[3]$.
      Then (since $\cE_2$ does not occur)
      there is some $u\in N(x)\cap V_s$ and $v\in N(y)\cap V_r$
      such that $\{u,v\}\in E$.
      Thus, $x$ and $y$ are connected in $G[V_r,V_s]$ by a path of length three. 
    \item[Case 2.] 
      Assume $x\in V_r\cap A_i$ and $y\in V_r\cap A_j$ for some $i\neq j$.
      Then (since $\cE_1$ does not occur) there is some $z\in N(y)\cap V_s\cap A_i$.
      By Case~1, there is a path in $G[V_r,V_s]$ between $x$ and $z$.
      Thus, there is a path in $G[V_r,V_s]$ between $x$ and $y$.
    \item[Case 3.] 
      Assume $x\in V_r\cap A_i$ and $y\in V_s\cap A_j$ for some $i\neq j$.
      Let $k$ be the unique element in $[3]\sm \{i,j\}$.
      Then (since $\cE_1$ does not occur)
      there is some $z\in N(y)\cap V_r\cap A_k$.
      By Case~2, there is a path in $G[V_r,V_s]$ between $x$ and $z$.
      Thus, there is a path in $G[V_r,V_s]$ between $x$ and $y$.
    \item[Case 4.] 
      Assume $x,y\in V_r\cap A_i$ for some $1\leq i\leq 3$.
      Let $j\neq i$.
      Then (since $\cE_1$ does not occur)
      there is some $z\in N(y)\cap V_s\cap A_j$.
      By Case~3, there is a path in $G[V_r,V_s]$ from $x$ to $z$.
      Therefore, there is a path in $G[V_r,V_s]$ from $x$ to $y$.
  \end{description}
  We conclude that $G[V_r,V_s]$ is connected.
  Hence, $(V_1,\ldots,V_d)$ is a strong $d$-rigid partition for $G$.
\end{proof}

\subsubsection{Putting it all together}
To conclude, we use the following simple lemma.
\begin{lemma}\label{lem:comb}
  Let $0\le 2d\le n$,
  and let $G=(V,E)$ be an $n$-vertex graph with $\delta(G)\ge (n+2d)/2-1$.
  Then, for every $A$ with $2d\le |A|<n$
  there exists $x\in V\sm A$
  such that $\deg(x,A)\ge d$.
\end{lemma}

\begin{proof}
  Set $a=|A|$.   If $a\ge n/2$
  then for every $x\notin A$,
  $\deg(x,A)\ge\delta(G)-(n-a-1)\ge d$.
  Hence, assume $2d\le a<n/2$.
  In this case,
  \begin{align*}
    |E(A,V\sm A)|
    &\ge a\left(\delta(G)-(a-1)\right)
    \ge a(n/2+d-a)
    \\ & =an/2+ad-a^2
  = a(n-2a)/2 +ad \ge d(n-2a)+ad =d(n-a).
  \end{align*}
  Thus, by the pigeonhole principle,
  there exists $x\in V\sm A$ with at least $d$ neighbours in $A$.
\end{proof}

Let us mention that an argument equivalent to \cref{lem:comb}
was used in the proof of Theorem~36 in~\cite{Jac16}.
We are ready to prove \cref{prop:approx:far}.

\begin{proof}[Proof of \cref{prop:approx:far}]
  Let $\delta,\alpha,n_0$ be the constants (depending on $\beta$) guaranteed by \cref{cor:supertriple},
  and let $n\ge \max\{n_0,2/\alpha\}$. Let $c'=c'(\delta)$ be the constant guaranteed by \cref{prop:triangle}.
  Set $c=\alpha c'/2$, and let $1\le d\le cn/\log{n}$.
  Let $G=(V,E)$ be an $n$-vertex graph with $\delta(G)\ge (n+2d)/2-1$
  which is $\beta$-far from being bipartite.
  We want to show that $G$ is $d$-rigid.
  By \cref{cor:supertriple},
  since $|E(G)|\ge n^2/4$ and since $G$ is $\beta$-far from being bipartite,
  there exist pairwise disjoint vertex sets $A_1,A_2,A_3\subseteq V$
  with $|A_1|=|A_2|=|A_3|=N\ge \alpha n\ge 2$
  and a subgraph $G''\subseteq G$
  such that the pairs $(A_i,A_j)$, $1\le i<j\le 3$
  are all $(\delta,2\delta)$-super-regular in $G''$.
  Let $H=G''[A_1\cup A_2\cup A_3]$ be the induced tripartite graph.
  By \cref{prop:triangle},
   $H$ admits a strong $D$-rigid partition
  for $D=\floor{c'N/\log{N}}\ge c n/\log{n}\ge d$.
  In particular, by \cref{thm:strong_rigid_partition}, $H$ is a $d$-rigid subgraph of $G$ with $3N\ge 2d$ vertices.
  Let $A$ be a largest vertex set of $G$ for which $G[A]$ is $d$-rigid.
  Assume to the contrary that $A\ne V$.
  Then, since $|A|\ge 2d$, by \cref{lem:comb}, there exists $x\in V\sm A$
  such that $\deg(x,A)\ge d$.
  Thus, by \cref{lem:extension}, $G[A\cup\{x\}]$ is also $d$-rigid, a contradiction.
\end{proof}

\section{Concluding remarks}\label{sec:concluding}

In this work, we studied the minimum degree conditions that guarantee the generic rigidity of a graph in $\mathbb{R}^d$. For small values of $d$ ($d=O(\sqrt{n})$), we established that every $n$-vertex graph with minimum degree at least $(n+d)/2 - 1$ is $d$-rigid (\cref{thm:exact}), matching the known $d$-connectivity threshold and therefore confirming \cref{conj:main} in that range. For larger values of $d$ ($d = O(n/\log^2 n)$), we showed that the minimum degree condition $(n+2d)/2 - 1$ ensures $d$-rigidity (\cref{thm:approx}), which is tight up to a factor of two in the coefficient of $d$.

Furthermore, as a corollary of our proof, we obtained a new lower bound on the pseudoachromatic number of a graph, depending on its minimum degree (\cref{thm:pseudo}), which may be of independent interest.

Despite this progress, \cref{conj:main}, which states that 
 the minimum degree threshold for $d$-rigidity is $\max\{(n+d)/2 - 1,\,2d - d(d+1)/n\}$, remains open. Proving it in full generality may require new techniques, particularly in the challenging regime where $d$ is linear in $n$. We hope that the methods and ideas presented here will stimulate further research on these and related questions.

\paragraph*{Acknowledgements}
We thank Yuval Peled for his helpful discussions and for his contribution to the formulation of \cref{lem:simpvx}.

\bibliography{library}
\end{document}